\documentclass{article}
\usepackage{amsfonts,amssymb,amsthm,amsmath,amsrefs}
\usepackage{enumerate}

\theoremstyle{plain}
\newtheorem{theorem}{Theorem}[section]

\newtheorem{thm}{Theorem}

\newtheorem{lemma}[theorem]{Lemma}
\newtheorem{corollary}[theorem]{Corollary}
\newtheorem*{Cor*}{Corollary}
\newtheorem{proposition}[theorem]{Proposition}

\theoremstyle{definition}
\newtheorem{definition}[theorem]{Definition}
\newtheorem{example}[theorem]{Example}

\theoremstyle{remark}
\newtheorem{remark}[theorem]{Remark}

\numberwithin{equation}{section}

\newcommand\NN{\mathbb{N}}
\newcommand\ZZ{\mathbb{Z}}

\newcommand\RR{\mathbb{R}}
\newcommand\CC{\mathbb{C}}
\newcommand\KK{\mathbb{K}}

\newcommand\g{\mathfrak g}

\newcommand\z{\mathfrak z}
\newcommand\hc{\mathfrak{hc}}
\newcommand\hcn{\mathfrak{hc}(V)}

\newcommand\C{\mathcal C}
\newcommand\U{\mathfrak U}
\newcommand\supp{\mathrm{supp}}
\DeclareMathOperator\HC{\mathrm{HC}}
\DeclareMathOperator\Ad{\mathrm{Ad}}
\DeclareMathOperator\End{\mathrm{End}}

\DeclareMathOperator\Hom{\mathrm{Hom}}
\DeclareMathOperator\id{\mathrm{id}}

\begin{document}
\title{Harmonic Analysis on Heisenberg--Clifford Lie Supergroups}
\author{Alexander Alldridge, Joachim Hilgert, Martin Laubinger}
\maketitle
\date

%


\begin{abstract}
We define a Fourier transform and a convolution product for functions and distributions on Heisenberg--Clifford Lie supergroups. The Fourier transform exchanges the convolution and a pointwise product, and is an intertwining operator for the left regular representation. We generalize various classical theorems, including the Paley--Wiener--Schwartz theorem, and define a convolution Banach algebra.
\end{abstract}






\section{Introduction} \label{sct:Intro}
 
Recently, several attempts have been made to extend the notion of a Fourier transform to a supersymmetric context. An obstacle one encounters is that supercommutative Lie supergroups with non-trivial odd part do not admit enough unitary representations to decompose reasonable spaces of superfunctions.

In this paper we associate a natural Fourier transform with purely odd Lie supergroups. The situation is analogous to  geometric quantization of translation groups on vector spaces: The action on the cotangent bundle is Hamiltonian only after central extension. This leads to the canonical commutation relations and thus, to the definition of the Heisenberg group. Similarly, the central extension of the underlying supergroup of a super vector space produces the Heisenberg--Clifford supergroup which admits unitary representations in abundance. In fact, the representation theory of Heisenberg--Clifford algebras resembles the representation theory of Heisenberg groups, so that the harmonic analysis of phase space can be used as a guideline for the harmonic analysis of purely odd superspaces. 

In the present paper, we will restrict ourselves to the case of a purely odd super vector space and its central extension. In a later paper, we will combine the purely odd with the classical, purely even Fourier analysis to provide a complete picture of the harmonic analysis for Heisenberg--Clifford supergroups.


To put our work in perspective, we mention some previous work on Fourier transforms of functions on linear supermanifolds. The earliest reference known to the authors is the article \cite{Rempel} by Rempel and Schmitt. More recent papers include work by Brackx, De Schepper and Sommen~\cite{Brackx}, and by De Bie~\cite{DeBie}. Investigation of Fourier transform on Heisenberg--Clifford Supergroups was started by Bieliavsky, de Goursac and Tuynman in their preprint~\cite{Bieliavsky}. These approaches have in common that the Fourier transform is defined in close formal analogy with the formula
\[
\hat f(\zeta) = \int_{\RR} f(x) e^{-i \zeta x} \, \mathrm{d}x,
\]
and one of the crucial ideas is an appropriate generalization of the exponential $e^{i \zeta x}.$ Our approach is somewhat different, in that we take as a starting point the formula
\[
\hat f(\pi) = \int_G f(x) \pi(x) \, \mathrm{d} x,
\]
where $\pi$ is an irreducible representation of $G,$ and $\mathrm{d} x$ is a left Haar measure on $G.$ Of course, for $G =\RR$ the two formulas agree if we take for $\pi$ the unitary character $x \mapsto e^{-i\zeta x}.$ Our approach is naturally covariant and thus well-adapted to the supergroup structures. Indeed, we expect that it will generalize to arbitrary Lie supergroups (though, of course, these may not have any unitary representations in general). The irreducible unitary representations of the Heisenberg--Clifford Lie supergroup have been classified by Salmasian~\cite{Salm}, and our approach depends heavily on his classification.

The Fourier transform $\mathcal F$ that we introduce takes values in a certain endomorphism algebra $\mathcal H.$ We define a convolution product in analogy with 
\[
(f*g)(x) = \int_G f(y)g(y^{-1}x) \, \mathrm{d}y,
\]
which behaves well with respect to the Fourier transform in the sense that the Fourier transform $\mathcal F(F*G)$ agrees with the pointwise product $\mathcal F(F) \mathcal F(G)$ in $\mathcal H.$ This convolution product seems to be new. (However, Bieliavsky {\it et. al.} \cite{Bieliavsky} define a $\star$-product which is mapped to a pointwise product under a quantization map).


We give a brief summary of the paper and state the main results. 

In Section~\ref{sct:Prelim} we define the Heisenberg--Clifford Lie algebra $\hc = \hc(V, \beta)$ associated with a symplectic super vector space $(V, \beta).$ We restrict our attention to $V$ purely odd and $\beta$ positive definite and non-degenerate. Then we define a Lie supergroup $\HC=(\HC_0, \mathcal C^\infty_{\HC})$ with underlying Lie group $\HC_0 = \RR,$ and Lie superalgebra $\hc.$ We introduce a left invariant integral $\int_{\HC} F$ for smooth compactly supported functions $F \in \C_c^\infty(\HC)$ and define distributions on $\HC.$ This material is known, except possibly for our treatment of the invariant integral. If $(a_i)_{i=1}^n$ is an orthonormal basis of $(V, \beta),$ then $\gamma := a_1 \cdots a_n \in \U(\hc)$ depends only on the orientation of the basis. We define $ \int_{\HC}F := \int_\RR F(\gamma;x) \, \mathrm{d} x$ - in our case, this is just the well-known Berezin integral, but an appropriate choice of $\gamma$ will yield an invariant integral on more general supergroups. The invariant integral gives rise to a non-degenerate invariant pairing $\langle \cdot, \cdot \rangle$ between $\C^\infty(\HC)$ and $\C_c^\infty(\HC).$ Lastly, we define spaces $\mathcal D'(\HC)$ and $\mathcal E'(\HC)$ of distributions and compactly supported distributions as topological dual spaces. The non-degenerate pairing then allows us to identify a smooth function $F$ with the distribution $\Phi \mapsto \langle F, \Phi \rangle.$

The definition of a finite-dimensional unitary representation of $\HC$ is given in Section~\ref{sct:Rep}.  It is known that all irreducible unitary representations of $\HC$ are finite-dimensional if $V$ is purely odd. Representations of the universal enveloping algebra $\U(\hc)$ in which a central element $z$ acts by a scalar $i \zeta$ factor through a Clifford algebra $Cl(V_\CC, \zeta \beta).$ The spin module of this algebra is then used to define a representation $(\pi_\zeta, \mathcal S)$ whenever $\zeta \in \CC;$ real and positive values of $\zeta$ yield all unitary irreducible representations of $\HC,$ which follows from the results of Salmasian~\cite{Salm}. It is well-known that the Clifford algebra $Cl(V_\CC, \zeta \beta)$ is isomorphic to an algebra $\mathcal H$ of endomorphisms of $\mathcal S,$ and we define a trace $T$ and a sesquilinear form $\langle A | B\rangle = T(AB^\dagger)$ on $\mathcal H.$ 

In Section~\ref{sct:FT}, we combine the invariant integral with the family $(\pi_\zeta, \mathcal S)$ of representations in order to define the Fourier transform of $F$ at $\zeta \in \CC$ to be the $\mathcal H$-valued integral $\mathcal F(F)(\zeta) := \widehat F(\zeta) := \int_{\HC} F \cdot \pi_{-\zeta}.$ The first main theorem is the following:
\begin{thm}
The Fourier transformation intertwines the left regular representation with $\pi_{-\zeta},$ that is,
\[
(L_{u;x}F) \widehat \;(\zeta) = \pi_{-\zeta}(u;x) \widehat{F}(\zeta)
\]
for $\zeta \in \CC, u \in \U(\hc)$ and $x \in \RR.$ 
\end{thm}

We introduce the Schwartz space $\mathcal S(\HC)$ as the space of functions $F \in \C^\infty(\HC)$ for which $F(u)$ is rapidly decreasing for all $u \in \U(\hc).$ Then the preceding theorem suggests a definition of a space $\mathcal S(\RR, \mathcal H)$  of $\mathcal H$-valued Schwartz functions on $\RR.$ The main idea is to define the components $A(u;\zeta)$ of a function $A: \RR \rightarrow \mathcal H$ in such a way that if $A = \widehat F,$ then $A(u)$ is the Fourier transform of $F(u).$ Then we can prove that the Fourier transform is an isomorphism of these topological vector spaces.
\begin{thm}
The Fourier transform restricts to an isomorphism of the topological vector spaces $\mathcal S(\HC)$ and $\mathcal S(\RR, \mathcal H).$
\end{thm}

Next, we extend the definition of Fourier transform to compactly supported distributions. If $U \in \mathcal D'(\HC),$ then its Fourier transform $\widehat U(\zeta)$ extends to an entire holomorphic function on $\CC.$ Lastly, we prove a Paley--Wiener--Schwartz theorem, characterizing the image of the space $\C^\infty_{[-a,a]}(\HC)$ of functions with support in the compact interval $[-a,a]$ under the Fourier transform.
\begin{thm}
The Fourier transform is a bijection between the space $\C^\infty_{[-a,a]}(\HC)$ and the space of functions $A: \CC \rightarrow \mathcal H$ whose components satisfy the following exponential growth condition:

For every $N \in \NN$ there is a constant $C_N$ such that
\begin{equation*}
|T(A(\zeta) d\pi_{-\zeta}(u)))| \leq C_N (1+|\zeta|)^{-N} e^{a \mathrm{Im}(\zeta)} \quad \text{ for all} \; u \in \U(\hc), \, \zeta \in \CC.
\end{equation*}
\end{thm}

Let $(m,m^*)$ and $(i,i^*)$ denote the multiplication and inversion morphism of the Lie supergroup $\HC.$ The last section begins with the definition of a convolution product 
\[
(F*G)(u;x) := (-1)^{|u|(|G|+|\gamma|)} \langle F,L_{u;x} i^* G \rangle
\]
where $F$ and $G$ are smooth functions on $\HC,$ one of which is compactly supported. In the rather technical Proposition~\ref{prop:ConvExist} we prove that our formula indeed yields a smooth function on $\HC,$ and that
\[
\langle F*G, \Phi \rangle = \langle F \otimes G, m^*\Phi \rangle = \langle F, i^*(G*i^*\Phi) \rangle.
\]
In Theorem~\ref{thm:ConvFT} we prove the following property of the convolution product.
\begin{thm}
If $F$ and $G$ are smooth compactly supported functions on $\HC,$ then
\[
(F*G) \widehat \;(\zeta)= \widehat F(\zeta) \widehat G(\zeta),
\]
\end{thm}

The convolution product can be extended to include convolutions $U*F$ of a distribution $U$ and a smooth function $F,$ if one of $F,U$ is compactly supported. Lastly, we define Sobolev-type spaces $(W^{k,p}(\HC), \| \cdot \|_{k,p}).$ If the order of differentiability $k$ is large enough, the convolution product can be extended to these Banach spaces, and we prove
\begin{thm}
If $n = \dim V,$ the space $W^{n,1}(\HC)$ is a Banach algebra with respect to the convolution product.
\end{thm}

We view the present set of results as a first step towards a systematic harmonic analysis on abelian Lie supergroups. We expect such a theory to have immediate applications to linear differential equations on superspaces. Moreover, it will be an important tool in a non-abelian harmonic analysis of homogeneous superspaces which is just evolving (see~\cites{AHZ:10,AlHi:10,Alld:10}).

\medskip\noindent\emph{Acknowledgements.} 
The first named author was funded by the Leibniz independent junior research group grant, and the SFB/Transregio 12 grant, both provided by Deutsche Forschungsgemeinschaft (DFG).

\section{Preliminaries} \label{sct:Prelim}

In this section we provide the basic definitions necessary for our construction. The material in this section is mostly known, so we will omit proofs  wherever possible. As general references, we mention \cites{Leites, DeligneMorgan}.


If $V = V_0 \oplus V_1$ is a super vector space, we write $|v|$ for the \emph{parity} of a homogeneous element $v \in V.$ If $m = \dim V_0$ and $n = \dim V_1,$ we say that $V$ is of \emph{graded dimension} $(m,n).$ If $W$ is another super vector space, we equip $V \otimes W$ with a grading such that $|v \otimes w| \equiv |v|+|w|(2).$ The space of all linear maps from $V$ to $W$ is denoted $\underline{\Hom}(V,W),$ and has grading defined in such a way that $|\phi(v)|\equiv|\phi|+|v|(2)$ for $\phi \in \underline{\Hom}(V,W), \, v \in V.$ We let $\Hom(V,W)$ denote the subspace of even linear maps. A bilinear form $\beta$ on $V$ is \emph{even} if $|u|+|v|=1$ implies $\beta(u,v)=0,$ and a non-degenerate even bilinear form $\beta$ is \emph{symplectic} if
\[
 \beta(u,v) = -(-1)^{|u||v|} \beta(v,u)
\]
for all homogeneous $u,v \in V.$ 

\begin{definition}
Given a finite-dimensional super vector space $V$ over $\RR$ together with a symplectic form $\beta$ on $V,$ we define the \emph{Heisenberg--Clifford Lie superalgebra} by 
\[
\hc(V,\beta)= V \oplus \RR,
\] 
with grading $\hc(V, \beta)_0 = V_0 \oplus \RR$ and $\hc(V,\beta)_1 = V_1$ and elements $u+x$ with $u \in V$ and $x \in \RR.$ We denote by $z$ the central element $0+1.$ The bracket is given by
\begin{equation}
[u+\lambda z, v+\mu z] = 2\beta(u,v)z, \; u,v \in V, \; \lambda, \mu \in \RR,
\end{equation}
and the one-dimensional center $\z(\hc(V,\beta))$ of $\hc(V,\beta)$ is spanned by $z.$
\end{definition}

\begin{remark}
Throughout this article, we will assume that $V$ is purely odd super vector space of graded dimension $(0,n).$ Then, $\beta$ is simply a non-degenerate symmetric bilinear form, and we assume that $\beta$ is positive definite. We will write $\hc$ or $\hc(V)$ for $\hc(V,\beta),$ and similarly $\z$ for $\z(\hc(V,\beta)),$ if no confusion is possible.
\end{remark}

\begin{remark}
Let $(a_i)_{i=1}^n$ be an orthonormal basis of $(V,\beta).$ Then the universal enveloping algebra $\U(\hc(V,\beta))$ is generated by the elements $a_i$ and $z \in \z(\hc(V,\beta)),$ subject to the relations 
\begin{equation}
a_i a_j = \begin{cases} - a_j a_i & \text{if} \quad i \neq j \\
z   & \text{if} \quad i=j \end{cases}
\end{equation}
and $z a_i = a_i z$ for all $i = 1, \ldots ,n.$
\end{remark}

As in the ungraded case, there is a \emph{symmetrization map} $\omega: S(\hc) \rightarrow \U(\hc).$ Here, $S(\hc) \cong \RR[z] \otimes \Lambda V$
is the symmetric algebra of the super vector space $\hc.$ The elements $a_i \in V$ pairwise anticommute, and therefore the map $\omega$ is simply given by $\omega(z^k \otimes (a_{i_1} \wedge \cdots \wedge a_{i_k})) = z^ka_{i_1} \cdots a_{i_k}.$
 
Given a natural number $n \geq 1,$ we let $\underline n = \{ 1, 2 , \ldots, n\}.$ If $(a_i)_{i=1}^n$ is a basis of $V,$ then the subsets of $\underline n$ parametrize a basis $(a_I)_{I \subset \underline n}$ of $\Lambda V$ in the usual way by 
\[
a_I = a_{i_1} \wedge \cdots \wedge a_{i_k},
\]
where $I = \{ i_1 < \cdots < i_k \}.$ We denote the images of $a_I$ under $\omega$ also by $a_I$. Then a Poincar\'e--Birkhoff--Witt basis of $\U(\hc)$ is given by $\{ z^k a_I \, | \, k \in \NN, \, I \subset \underline n \}.$

We introduce a special element of $\hc(V, \beta),$ which is up to a sign the \emph{chirality operator} in the theory of Clifford algebras.

\begin{definition} \label{def:gamma}
If $(a_i)_{i=1}^n$ is an orthonormal basis of $(V, \beta),$ we let 
\[
\gamma: = a_1 \cdots a_n = \omega(1\otimes (a_1 \wedge \cdots \wedge a_n)) \in \U(\hc(V, \beta)).
\]
Since the volume element $a_1 \wedge \cdots \wedge a_n \in \Lambda^nV$ only depends on the orientation of the orthonormal basis, the same is true for $\gamma.$
\end{definition}

Recall that if $\g$ is a Lie superalgebra, then $\U(\g)$ carries the structure of a super Hopf algebra (see {\it e.g.}~\cite[Section 3]{Kostant}). The \emph{coproduct} $\Delta$ and the \emph{antipode} $S$ will be used below to define a Lie supergroup corresponding to $\hc.$ They are uniquely determined by the following properties: The coproduct $\Delta: \U(\g) \rightarrow \U(\g) \otimes \U(\g)$ is an even unital algebra homomorphism which satisfies $\Delta(x) = x \otimes 1 + 1 \otimes x$ for $x \in \g.$ The antipode $S: \U(\g) \rightarrow \U(\g)$ is an even unital super-antiautomorphism, that is, $S(uv) = (-1)^{|u||v|}S(v)S(u),$ and on elements $x \in \g$ it is given by $S(x) = -x.$

\begin{remark} \label{rem:deltagamma} 
The element $\Delta(\gamma)$ will play an important role in this article. With respect to an orthonormal basis $(a_i)_{i=1}^n,$ it is given by 
\begin{equation}
\Delta(\gamma) = \prod_{i=1}^n (a_i \otimes 1 + 1 \otimes a_i) = \sum_{I \subset \underline n} a_I \otimes *a_I.
\end{equation}
Here $*a_I = \pm a_{I^c},$ where the sign is such that $a_I \cdot (*a_I) = \gamma$ in $\U(\hc),$ and $I^c = \underline n \setminus I.$ Concretely, if $I = (i_1 < \ldots < i_k),$ let $\sigma_I$ denote the permutation of $\underline n$ determined by $\sigma_I(j) = i_j$ for $1 \leq j \leq k$ and $\sigma_I(k+1) < \ldots < \sigma_I(n).$ Then 
\[
*a_I = \mathrm{sgn}(\sigma_I) a_{I^c}.
\]
Note that both $(a_I)_{I \subset \underline n}$ and $(*a_I)_{I \subset \underline n}$ form a bases of $\U(\hc)$ as a $\z$-module.
We will use Sweedler's notation 
\begin{equation} \label{eq:Sweedler}
\Delta(\gamma) = \sum_i \gamma_i^{(1)} \otimes \gamma_i^{(2)}.
\end{equation}
This, however, requires some care, since the $\gamma_i^{(j)}$ are not uniquely determined by equation~\eqref{eq:Sweedler}.
\end{remark}


In \cite{Koszul}, Koszul constructs a Lie supergroup associated with a Lie supergroup pair. We recall the definition of a Lie supergroup pair and the construction of the corresponding sheaf.

\begin{definition}
A \emph{Lie supergroup pair $G=(G_0,\g)$} consists of a Lie group $G_0,$ and a real Lie superalgebra $\g$ whose even part $\g_0$ is the Lie algebra of $G_0,$ and a smooth linear action $\Ad$ of $G_0$ on $\g$ by even linear automorphisms. We require that the action $\Ad$ extends the adjoint action of $G_0$ on $\g_0$ and that its differential $d\Ad: \g_0 \times \g \rightarrow \g$ is the restriction of the bracket $[\cdot, \cdot].$ The subalgebra $\g_0$ of $\g$ acts on $\U(\g)$ from the left, and if $U \subset G_0$ is open, then $\g_0$ acts on $C^\infty(U)$ by left invariant differential operators. Consider $C^\infty(U)$ as a purely even $\g_0$-module and define
\[
\C^\infty_G(U) := \underline{\Hom}_{\g_0}(\U(\g),C^\infty(U)).
\]
If $F \in \C^\infty_G(U)$ we write $F(u;x)$ for $F(u)(x)$ if $u \in \U(\g)$ and $x \in U.$
\end{definition}

\begin{definition}
We define the Heisenberg--Clifford Lie supergroup pair by $HC = (\RR, \hc),$ where the action $\Ad$ of $\RR$ on $\hc$ is the trivial action. If $U \subset \RR$ is open, we let $z \in \z$ act on $C^\infty(U)$ by $zf = -f'.$
\end{definition}

\begin{proposition} \label{prop:SheafAlg}
\begin{enumerate}[a)]
\item Let $U \subset \RR$ be open, and denote by $\mu$ the pointwise multiplication of functions in $C^\infty(U).$ The assignment $U \mapsto \C^\infty_{\HC}(U)$ is a sheaf of supercommutative unital superalgebras on $\RR,$ if the algebra multiplication is defined by 
\[
F \cdot G := \mu \circ (F \otimes G) \circ \Delta.
\]
The pair $(\RR, \C^\infty_{\HC})$ is a supermanifold.
\item  As a superalgebra, $\C^\infty_{\HC}(U)$ is isomorphic to 
\[
\underline{\mathrm{Hom}}_\RR (\Lambda V, C^\infty(U)) \cong C^\infty(U) \otimes \Lambda V^*,
\]
where the algebra structure on the right hand side is the obvious one. The isomorphism is given by $F \mapsto (F \circ \omega)|_{\Lambda V},$ where $\omega$ is the symmetrization map.
\end{enumerate}
\end{proposition}

\begin{remark}
If $(a_i)_{i=1}^n$ is a basis of $V,$ let $(\xi^i)_{i=1}^n$ denote the dual basis of $V^*.$ We use superscripts $\xi^I = \xi^{i_1} \wedge \ldots \wedge \xi^{i_k}$ for the elements of $\Lambda V^*,$ since this has become standard in the literature on supermanifolds. Due to the simple form of the symmetrization map $\omega,$ the isomorphism in Proposition~\ref{prop:SheafAlg} b) is given by 
\[
F \mapsto \sum_{I \subset \underline n} f_I \otimes \xi^I,
\]
where $f_I := F(a_I).$ This coordinate-dependent notation for smooth functions is quite common in the literature. However, we will avoid using coordinates as far as possible.
\end{remark}

We recall that a \emph{Lie supergroup} is a supermanifold $(G, \C^\infty_G)$ together with morphisms $m=(m_0,m^*), i=(i_0,i^*)$ and $e: * \rightarrow (G, \C^\infty_G),$ satisfying the usual group axioms (here, $*$ is the $(0,0)$-dimensional supermanifold).

\begin{proposition}
The supermanifold $(\RR, \C^\infty_{\HC})$ is a Lie supergroup with multiplication
\[
m= (m_0, m^*), \quad m_0(x,y) := x+y,  \quad (m^*F)(u\otimes v; x, y) := F(uv; x+y),
\]
inversion 
\[
i=(i_0, i^*), \quad i_0(x):=-x, \quad (i^*F)(u;x) := F(S(u);-x),
\]
and identity element $e=(e_0,e^*)$ given by $e_0(*):=0$ and $e^*F := F(1;0).$
\end{proposition}
\begin{definition}
We denote the algebra of global sections by $\C^\infty(\HC):= \C^\infty_{\HC}(\RR)$ and refer to elements of $\C^\infty(\HC)$ as \emph{smooth functions on $\HC$}.

The \emph{left regular action} of $\HC$ on $\C^\infty(\HC)$ is given by 
\[
(L_xF)(u;y) := F(u;y-x), \quad (L_uF)(v,y) := (-1)^{|u||F|}F(S(u)v;y)
\]
for $x \in \RR$ and $u \in \U(\hc).$ This defines linear maps $L_x, L_u: \C^\infty(\HC) \rightarrow \C^\infty(\HC)$ of parity $|L_x|=0$ for $x \in \RR$ and $|L_u|=|u|$ for $u \in \U(\hc).$ We write $L_{u;x}$ for $L_u \circ L_x = L_x \circ L_u.$
\end{definition}

\begin{lemma}
\begin{enumerate}[a)]
\item The assignments $x\mapsto L_x$ and $u \mapsto L_u$ define representations of $\RR$ and $\U(\hc)$ on the vector space $\C^\infty(\HC).$

\item If $v \in V,$ then $L_v$ is a super-derivation on $\C^\infty(\HC),$ that is, $L_v(F \cdot G) = L_vF \cdot G + (-1)^{|v||F|}F \cdot L_vG.$
\end{enumerate}
\end{lemma}

\begin{definition}
By Proposition~\ref{prop:SheafAlg}, $\C^\infty(\HC)$ is isomorphic as a superalgebra to $C^\infty(\RR) \otimes \Lambda V^*.$ Since $C^\infty(\RR)$ carries a nuclear Fr\'echet topology, this tensor product also carries a nuclear Fr\'echet topology. For each compact $K \subset \RR$ and $u \in \U(\hc)$ we define a seminorm on $\C^\infty(\HC)$ by 
\[
p_{K,u}(F):= \max_{x \in K} |(L_uF)(1;x)|
\]
Given a basis $(a_i)_{i=1}^n$ of $V$ and a countable exhaustion $\{K_j\}_{j \in J}$ of $\RR$ by compact sets, the Fr\'echet topology on $\C^\infty(\HC)$ can be defined by the countable family $\{ p_{K_j, z^ka_I} \}_{j,k,I}$ of seminorms.
\end{definition}

We define vector valued and compactly supported functions as well as functions of Schwartz class. 

\begin{definition} \label{def:SchwartzSpace}
If $K \subset \RR$ is compact, we let
\[
\C_K^\infty(\HC) := \{ F \in \C^\infty(\HC) \;|\;  (\forall \, u \in \U(\hc)) \,:\, \supp \, F(u) \subset K \}
\]
be the space of smooth functions with support contained in $K,$ and we give $\C^\infty_K(\HC)$ the topology defined by the seminorms $p_u(F) = \max_{x \in K} |L_uF(1;x)|.$ Then the union 
\[
\C_c^\infty(\HC):= \cup_i \C^\infty_{K_i}(\HC),
\]
where $\{K_i\}$ is a countable exhaustion of $\RR$ by compact sets, is the \emph{space of compactly supported smooth functions on $\HC,$} which is a countable strict inductive limit of Fr\'echet spaces, or an LF space. 

If $W$ is a finite-dimensional super vector space over $\RR$ or $\CC,$ we define the vector space of \emph{smooth $W$-valued functions on $\HC$} by $\C^\infty(\HC,W):= (\C^\infty(\HC) \otimes W)_0.$ 

Lastly, the \emph{Schwartz space} $\mathcal S(\HC)$ of rapidly decreasing functions is defined to be the space of $F \in \C^\infty(\HC)$ for which
\[
s_{j,u}(F) := \sup_{x \in \RR} |x^j (L_uF)(1;x)| < \infty
\] 
for all $j \in \NN$ and $u \in \U(\hc).$
\end{definition}

\begin{remark}
\begin{enumerate}[a)]
\item  The space $\mathcal S(\HC)$ is simply the subspace of $\C^\infty(\HC)$ of all $F$ which satisfy $F(u) \in \mathcal S(\RR)$ for all $u \in \U(\hc).$ 
\item  The spaces of functions we have defined so far are isomorphic as vector spaces to $C^\infty(\RR) \otimes W,\,  C_c^\infty(\RR) \otimes W$ and $\mathcal S(\RR) \otimes W,$ respectively, where $W = \Lambda V^*$ is finite-dimensional. Therefore, there is only one reasonable tensor product topology, and we will use this topology throughout.
\end{enumerate}
\end{remark}

\begin{lemma}
The linear maps $L_{u;x}$ are continuous on $\C^\infty(\HC), \,\C_c^\infty(\HC)$ and  $\mathcal S(\HC).$
\end{lemma}
\begin{proof}
After choosing coordinates, the proof reduces to showing that the derivative $f \mapsto f'$ is continuous on $C^\infty(\RR), \, C^\infty_c(\RR)$ and $\mathcal S(\RR),$ which is trivial by definition.
\end{proof}

\subsection*{The Invariant Integral}

\begin{definition} \label{def:InvInt}
Let $W$ be a finite-dimensional super vector space.
\begin{enumerate}[a)]
\item If $F \in \C^\infty_c(\HC,W),$ we define the \emph{integral of $F$ over $\HC$} as
\[
\int_{\HC} F := \int_\RR F(\gamma;x) \, \mathrm{d}x,
\]
where $\gamma$ is defined in~\ref{def:gamma}.
\item If $F \in \C^\infty(\HC)$ and $G \in \C^\infty(\HC,W)$ are  such that $F\cdot G$ has compact support, we let
\[
\langle F,G \rangle := \int_{\HC} F\cdot G.
\]
\end{enumerate}
\end{definition}

\begin{remark} \label{rem:ParityPairing}
\begin{enumerate}[a)]
\item The integral and the pairing have parity $|\gamma|,$ that is, if $|F|+|\gamma| \equiv 1 (2),$ then $\int_{\HC} F= 0,$ and if $|F|+|G|+|\gamma| \equiv 1 (2) ,$ then $ \langle F,G \rangle =0.$
\item The product in $\C^\infty(\HC)$ is supercommutative, and therefore
\[
\langle F,G \rangle = (-1)^{|F||G|} \langle G,F \rangle.
\]
\end{enumerate}
\end{remark}

\begin{lemma}
The integral is left invariant in the sense that
\[
\int_{\HC} L_xF = \int_{\HC} F \quad \text{and} \quad \int_{\HC} L_uF = 0
\]
for all $x \in \RR$ and all $u \in \U(\hc).$ The pairing $\langle \cdot, \cdot \rangle $ is invariant in the sense that
\begin{equation} \label{eq:Invariance}
\langle L_{u;x}F,G \rangle = (-1)^{|F||u|} \langle F,L_{S(u);-x}G \rangle 
\end{equation}
for $x \in \RR$ and $u \in \U(\hc).$
\end{lemma}
\begin{proof}
Invariance under $L_x, x \in \RR$ follows from translation invariance of the Lebesgue measure, since $L_xF(\gamma;y) = F(\gamma;y-x).$ In order to check invariance under $\U(\hc),$ choose an orthonormal basis $(a_i)_{i=1}^n$ of $V.$ It then suffices to show that $\int L_{a_i}F = 0$ for $1 \leq i \leq n$ and $\int L_zF = 0.$ We compute
\begin{align*}
\int_{\HC} L_{a_i} F &= \pm \int_\RR F(a_i \gamma;x) \,\mathrm{d}x \\
 &= \pm \int_\RR F(z a_1 \dots \hat a_i \dots a_n;x) \,\mathrm{d}x \\
 &= \pm \int_\RR F(a_1 \dots \hat a_i \dots a_n )'(x) \,\mathrm{d}x 
\end{align*}
which is zero because $F(a_1 \dots \hat a_i \dots a_n)$ is compactly supported. For the same reason, $\int L_zF = \int_{\RR} F(\gamma)'(x) \, \mathrm{d} x = 0.$ If $v \in \hc,$ we have $\int_{\HC} L_v(F \cdot G) = 0,$ and because $L_v$ is a super-derivation, this implies
\[
\langle L_vF,G \rangle = -(-1)^{|v||F|} \langle F,L_vG \rangle.
\]
Let $u,v \in \hc.$ Then $-L_u = L_{-u} = L_{S(u)},$ and it follows that
\begin{align*}
\langle L_{uv}F,G \rangle =(-1)^{|F|(|uv|)} \langle F, L_{S(uv)}G \rangle.
\end{align*}
This implies that equation~\eqref{eq:Invariance} holds for arbitrary $u \in \U(\hc)$ and $x \in \RR.$
\end{proof}

\begin{lemma} \label{lem:Pairing}
If $F,G \in \C^\infty_c(\HC),$ then
\[
\left| \int_{\HC} F \right| \leq \mathrm{vol}(\supp F) \cdot p_{\supp F, \gamma}(F)
\]
and
\[
p_{K,u}(F \cdot G) \leq \sum_i p_{K, u_i^{(1)}}(F) p_{K, u_i^{(2)}}(G)
\]
where $u \in \U(\hc)$ and $\Delta(u)= \sum_i u_i^{(1)} \otimes u_i^{(2)}.$

In particular, the integral $\int_{\HC}$ is a continuous linear functional, and the algebra multiplication on $\C^\infty(\HC)$ and the pairing $\langle \cdot, \cdot \rangle $ are continuous.
\end{lemma}

Lastly, we define distributions and compactly supported distributions.

\begin{definition}
We define the spaces $\mathcal{D}'(\HC)$ of \emph{distributions on $\HC$} and $\mathcal{E}'(\HC)$ of \emph{compactly supported distributions on $\HC$} to be the topological dual spaces of $\C_c^\infty(\HC)$ and $\C^\infty(\HC),$ respectively. We introduce the \emph{duality pairing} $\langle \cdot, \cdot \rangle$ and write $\langle U, \Phi \rangle:= U(\Phi)$ if $U$ is a distribution and $\Phi$ is a smooth function.
\end{definition}

\begin{remark} \label{rem:Distributions}
a) By Lemma~\ref{lem:Pairing}, every element $F \in \C^\infty(\HC)$ defines a distribution via $\Phi \mapsto \langle F, \Phi \rangle,$ and the corresponding map $\C^\infty(\HC) \rightarrow \mathcal D'(\HC)$ is injective.  Similarly, there is an injection $\C^\infty_c(\HC) \rightarrow \mathcal E'(\HC).$

b) The spaces $\C^\infty_c(\HC)$ and $C^\infty_c(\RR) \otimes \Lambda V^*$ are isomorphic as algebras and as topological vector spaces. Hence, the topological dual $\mathcal D'(\HC)$ can be identified with $\mathcal D'(\RR) \otimes \Lambda V.$ For $U \in \mathcal D'(\HC)$ we denote $U(a_I) \in \mathcal D'(\RR)$ the distribution determined by $U( f \otimes \xi^I) = U(a_I)(f).$ Similarly, we define $U(a_I) \in \mathcal E'(\RR)$ if $U \in \mathcal E'(\HC).$
\end{remark}

\begin{lemma} \label{lem:PairExplicit}
Let $F,G \in \C^\infty(\HC),$ and assume that one of $F,G$ is compactly supported. Then
\[
\langle F,G \rangle= \sum_{i} \int_\RR (-1)^{|\gamma_i^{(1)}||\gamma_i^{(2)}|}F(\gamma_i^{(1)};x) G(\gamma_i^{(2)};x) \, \mathrm{d} x.
\]
Similarly, if $U \in \mathcal D'(\HC)$ or $U \in \mathcal E'(\HC),$ there are distributions  $U(\gamma_i^{(1)}) \in \mathcal D'(\RR)$ or in $\mathcal E'(\RR)$ such that 
\begin{equation} \label{eq:PairingDist}
\langle U, \Phi \rangle= \sum_i (-1)^{|\gamma_i^{(1)}||\gamma_i^{(2)}|} \langle U(\gamma_i^{(1)}), \Phi(\gamma_i^{(2)}) \rangle 
\end{equation}
for all $\Phi \in \C_c^\infty(\HC)$ or in $\C^\infty(\HC),$ respectively.
\end{lemma}
\begin{proof}
By definition, 
\begin{align*}
(F\cdot G)(\gamma;x) &= \mu((F \otimes G)( \Delta(\gamma);x,x)) \\
&= \sum_i \mu((F \otimes G)(\gamma_i^{(1)} \otimes \gamma_i^{(2)};x,x)) \\
&= \sum_i (-1)^{|\gamma_i^{(1)}||G|} F(\gamma_i^{(1)};x) G(\gamma_i^{(2)};x).
\end{align*}
Now  observe that $|G(\gamma_i^{(2)})| = |G|+|\gamma_i^{(2)}|=0,$ since $C^\infty(\RR)$ is purely even, and it follows that $|G| = |\gamma_i^{(2)}|.$ 

The fact that $\langle U,\Phi \rangle$ can be written as $\sum_i (-1)^{|\gamma_i^{(1)}||\gamma_i^{(2)}|} U(\gamma_i^{(1)})(\Phi(\gamma_i^{(2)}))$ follows from Remark~\ref{rem:Distributions} b).
\end{proof}


\section{Representations} \label{sct:Rep}

We define representations and unitary representations of Lie supergroup pairs. In this, we follow Alldridge~\cite[Appendix B]{Alld:10} and Carmeli {\it et.al.}~\cite{Vara}. Then we use spin modules to construct a family $(\pi_\zeta)_{\zeta \in \CC}$ of representations of $\HC.$ Salmasian showed in~\cite{Salm} that all irreducible unitary representations of $\HC$ are, up to unitary equivalence, of the form $\pi_\zeta$ with $\zeta$ real and positive. Since the representations $\pi_\zeta$ are a crucial ingredient in our definition of the Fourier transform, we give a detailed description.

\begin{definition}
Let $V$ be a finite-dimensional super vector space over $\KK \in \{ \RR, \CC\}$. A \emph{representation $\pi = (\pi_0, d\pi)$ of a Lie supergroup pair $G=(G_0,\g)$ on $V$} consists of a representation $\pi_0$ of $G_0$ on $V$ by even $\KK$-linear maps, and a Lie superalgebra representation $d\pi$ of $\hc$ on $V,$ such that $d(\pi_0) = d\pi|_{\g_0}$ and $d\pi(\Ad(g)x) = \pi_0(g) d\pi(x) \pi_0(g^{-1})$ for all $g \in G_0$ and $x \in \g.$ 
\end{definition}

The global functions $\mathbb{A} := \C^\infty_G(G_0)$ form a supercommutative $\RR$-superalgebra. If $V$ is a finite-dimensional super vector space, then so is $\underline{\End}(V),$ and we define $\C^\infty(G_0, \underline{\End}(V)) := (\mathbb A \otimes \underline{\End}(V))_0.$ This space can be identified with the space $\End_{\mathbb A}(\mathbb{A} \otimes V)$ of even $\mathbb{A}$-linear endomorphisms of the left $\mathbb{A}$-module $\mathbb{A} \otimes V.$ Consider the subset $\mathrm{GL}(\mathbb{A} \otimes V)$ of invertible $\mathbb{A}$-linear endomorphisms. We have the following characterization of linear representations of $\HC$ on $V.$

\begin{proposition} \label{prop:Rep}
Linear representations $\pi$ of $G$ on a finite-dimensional super vector space $V$ are in bijective correspondence with elements $F \in \mathrm{GL}(\mathbb{A} \otimes V) \subset \C^\infty(G_0,\underline{\End}(V))$ which satisfy
\[
(m^* \otimes \id_V) \circ  F = (\id_V \otimes F) \circ F \quad \text{and} \quad
(e^* \otimes \id_V) \circ F = \id_V.
\]
\end{proposition}
\begin{proof}
See~\cite[Proposition B.19]{Alld:10}. For later use, we just note that the element $F \in \C^\infty(G_0,\underline{\End}(V))$ corresponding to a representation $\pi=(\pi_0, d\pi)$ is given by
\begin{equation} \label{eq:Rep}
F(u;x) = \pi_0(x) \circ d\pi(u) \in \underline{\End}(V).
\end{equation}
\end{proof}

\begin{definition}
Let $(\mathcal H, (\cdot, \cdot))$ be a $\ZZ_2$-graded Hilbert space over $\CC.$ We say that $(\mathcal H, (\cdot, \cdot))$ is a \emph{super Hilbert space} if the graded pieces are orthogonal with respect to $(\cdot, \cdot).$ If $(\mathcal H, (\cdot, \cdot))$ is a super Hilbert space, we define the \emph{super inner product} by $\langle u |v \rangle := i^{|u||v|}(u,v),$ and the super adjoint $T^\dagger$ of a continuous linear operator by $T^\dagger  := (-1)^{|T|}T^*,$ where $T^*$ is the usual adjoint.
\end{definition}

\begin{remark}
The definitions of $\langle \cdot | \cdot \rangle$ and $T^\dagger$ are such that
\[
\langle u |v \rangle = (-1)^{|u||v|} \overline{\langle v |u\rangle} \quad \text{and} \quad \langle Tu |v \rangle = (-1)^{|u||T|} \langle u |  T^\dagger v \rangle.
\]
\end{remark}

\begin{definition}
A representation $\pi = (\pi_0, d\pi)$ of a Lie supergroup pair $(G_0, \g)$ on a finite-dimensional super Hilbert space is \emph{unitary} if $\pi_0$ is a unitary representation of $G_0$ and $d\pi(u)^\dagger = -d\pi(u)$ for all $u \in \g.$ 
\end{definition}

\begin{remark}
\begin{enumerate}[a)]
\item Observe that if $\pi$ is a unitary representation, then
\[
d\pi(u)^\dagger = d\pi(S(u))
\]
for all $u \in \U(\g).$ 
\item We restrict our attention to finite-dimensional representations because all irreducible unitary representations of $\HC$ are finite-dimensional. In the general setting, there are technical subleties due to the fact that the operators $d\pi(x), x \in \g_1$ are in general unbounded (see \cite[Definition 2]{Vara}).
\item Suppose that $\pi = (\pi_0, d\pi)$ is a unitary representation of $G=(G_0,\g).$ If we let $\rho(x) = e^{-i\pi/4}d\pi(x)$ for $x \in \g_1,$ then the $\rho(x)$ are self-adjoint and satisfy 
\[
\rho(x) \rho(y) + \rho(y) \rho(x) = -id\pi([x,y])
\]
(see~\cite[Section 2.3]{Vara} for details). 
\end{enumerate}
We will use this observation in the next subsection by first constructing operators $c_\zeta(v)$ for $v \in \hc_1 =V,$ which are self-adjoint if $\zeta$ is real and non-negative, and then setting
\[
d\pi_\zeta(v)= e^{i\pi/4} c_\zeta(v)
\]
for $v \in V.$
\end{remark}


\subsection*{Spin Modules}

The construction by Carmeli {\it et.al.}~\cite{Vara} and Salmasian~\cite{Salm} of unitary representations of $\HC$ is based on the following idea. If $\mathcal H$ is an irreducible unitary representation of $\HC,$ then by a super version of Schur's lemma, the central element $z$ acts by a scalar $i\zeta.$ This scalar $\zeta$ has to be positive, essentially because $z$ is the square of an odd element in $\hc.$ The operators $c_\zeta(v) = e^{-i\pi/4} d\pi(v)$ for $v \in V$ are self-adjoint and satisfy
\[
[c_\zeta(v),c_\zeta  w)] = 2 \zeta \beta(v,w) \id.
\]
This means that $c_\zeta$ factors through a representation of the quotient of the complexified universal enveloping algebra $\U(\hc)_\CC$ by the ideal generated by $z-\zeta.$ But this quotient is a Clifford algebra $Cl(V_\CC, \zeta \beta),$ and the irreducible representations of Clifford algebras are well-known.

We will need this construction also for general $\zeta \in \CC,$ in which case the corresponding representations are no longer unitary. Also, we need some refined information about the representation, and therefore we recall the construction in some detail. As additional references, we use the exposition by Deligne~\cite[Proposition 2.2]{Deligne}, and the book by Rosenberg~\cite[Section 2.2.2]{SRosenberg}.

\begin{proposition} \label{prop:AlgIso}
Consider the complex space $(V_\CC, \zeta \beta),$ where $\zeta$ is any non-zero complex number.
\begin{enumerate}[a)]
\item If $\dim V = 2k>0$ is even, then $Cl(V_\CC, \zeta \beta)$ is isomorphic as complex superalgebra to $\mathcal H = \underline \End(\mathcal S),$ where $\mathcal S$ is the complex super vector space $\CC^{N|N}, \; N=2^{k-1}.$
\item Let $D$ be the superalgebra $\CC[\epsilon]$ with $\epsilon$ odd and $\epsilon^2 = \zeta.$ If $\dim V = 2k+1,$ then $Cl(V_\CC, \zeta \beta)$ is isomorphic as complex superalgebra to $\mathcal H =\underline{\End}_D(\mathcal S),$ where $\mathcal S = D^N = D \otimes_\CC \CC^N, \, N = 2^k$ is a left $D$-module.
\end{enumerate}
\end{proposition}
\begin{proof}
We first consider the case $\dim V_\CC =2k >0.$ The choice of an orthonormal basis in $V$ yields a tensor product decomposition
\[
Cl(V_\CC, \zeta \beta) = Cl(\CC^2) \otimes \cdots \otimes Cl(\CC^2),
\]
(see \cite{Deligne}), where the spaces $\CC^2$ are equipped with the bilinear form $(u,v) = \zeta(u_1v_1 + u_2 v_2).$ Therefore, it suffices to consider the case $k=1.$ Let $a_1, a_2$ be an orthonormal basis of $V_\CC$ and let
\[
c_\zeta(a_1) := \begin{pmatrix} 0&\zeta \\ 1 & 0 \end{pmatrix}, \quad c_\zeta(a_2):= \begin{pmatrix} 0& i\zeta \\ -i & 0\end{pmatrix} \in \underline{\End}(\CC^{1,1}).
\]
This clearly defines a representation of $Cl(\CC^2)$ and the arguments in \cite{Deligne} show that $c_\zeta$ is an algebra isomorphism.

Now we turn to the case of $\dim V_\CC = 2k+1,$ and again we can reduce to $k=1.$ It suffices to construct elements $c_\zeta(a_i) \in \underline{\End}_D(D \otimes \CC^2)$ for a basis $a_0, a_1, a_2$ of $V_\CC.$ To this end, we follow~\cite{Vara}.

We let $x_0 = \epsilon \otimes \id, x_1 = 1 \otimes c_1(a_1)$ and $x_2= 1 \otimes c_1(a_2)$ in $D \otimes \mathrm{Mat}(2, \CC),$ where the $c_1(a_i)$ are defined as in the even case with $\zeta = 1.$ Now we let
\[
c_\zeta(a_0) = i x_0 x_1 x_2, \; c_\zeta(a_1) = -ic_\zeta(a_0) x_1, \;  c_\zeta(a_2) = -ic_\zeta(a_0) x_2
\]
(see~\cite{Vara}). Then a simple computation shows that this defines a representation of $Cl(V_\CC, \zeta \beta),$ and an isomorphism of superalgebras.
\end{proof}

\begin{definition}
We define the symbol $[n]$ for $n \in \NN$ by $[n]= n/2$ if $n$ is even and $[n] = (n+1)/2$ if $n$ is odd. Let $\mathcal H$ be as in  Proposition~\ref{prop:AlgIso}, and consider the linear functional $T$ on $\mathcal H$ defined by
\[
T(A) = \begin{cases} \mathrm{STr}(A), & n \; \text{even} \\
\mathrm{Tr}(e^{i\pi/4}\epsilon A), & n\; \text{odd} \end{cases}.
\]
We define a sesquilinear form on $\mathcal H$ by $\langle A| B \rangle := T(AB^\dagger).$
\end{definition}

\begin{lemma} \label{lem:Tracec}
Let $\mathcal S$ and $\mathcal H$ be as in Proposition~\ref{prop:AlgIso}.
\begin{enumerate}[a)]
\item If $\zeta$ is real, then $\mathcal S$ carries a sesquilinear form $(\cdot, \cdot)_\zeta$ such that $c_\zeta(v)$ is self-adjoint for $v \in V.$ The sesquilinear form is positive definite if $\zeta >0.$ If $\zeta<0$ the form is positive definite on $\mathcal S_0$ and negative definite on $\mathcal S_1.$ 
\item Let $(a_i)_{i=1}^n$ be an orthonormal basis of $V.$ Then
\[
T(e^{i n \pi/4} c_\zeta(\gamma)) = (2\zeta)^{[n]}
\]
and $T(c_\zeta(a_I)) = 0$ if $I \subsetneq \underline n.$ The sesquilinear form $\langle A | B \rangle = T(AB^\dagger)$ on $\mathcal H$ is non-degenerate, positive definite on $\mathcal H_0$ and negative definite on $\mathcal H_1.$ 
\end{enumerate}
\end{lemma}
\begin{proof}
a)
Assume that $\zeta$ is real. In the case of $n=2$ we define the sesquilinear form $(\cdot, \cdot)_\zeta$ on $\mathcal S = \CC^2$ by $(u,v)_\zeta = u_1 \bar v_1 + \zeta u_2 \bar v_2.$ Clearly, the $c_\zeta(a_i)$ are self-adjoint. Moreover, the inner product is negative definite on $\mathcal S_1$ if $\zeta <0.$ Taking tensor products yields the general result for the even case.

Now we consider the case when $n$ is odd. We identify $D \cong \CC^2,$ so that multiplication by $\epsilon$ has matrix $\begin{pmatrix} 0&\zeta \\ 1& 0\end{pmatrix}$ with respect to the standard basis of $\CC^2.$ If we define $(\cdot, \cdot)_\zeta$ on $D$ by $(u,v)_\zeta = u_1 \bar v_1 + \zeta u_2 \bar v_2,$ then multiplication by $\epsilon$ is self-adjoint. By definition (see proof of Proposition~\ref{prop:AlgIso}), the operators $c_\zeta(a_i)$ are of the form $\epsilon \otimes A$ where $A$ is self-adjoint with respect to the standard inner product on $\CC^{2^k}.$ Extending these inner products to $\mathcal S = D \otimes \CC^{2^k}$ makes all $c_\zeta(v), v \in V$ self-adjoint. Again, it is clear that if $\zeta <0,$ then the inner product is negative definite on $\mathcal S_1.$ 

b)
If $n=2,$
\[
c_\zeta(a_1a_2) = \begin{pmatrix} -i\zeta & 0 \\ 0 & i\zeta \end{pmatrix}  
\]
shows that $T(ic_\zeta(\gamma)) = 2\zeta.$ If $n=3,$ we have
\[
c_\zeta(a_0) = \epsilon \otimes \begin{pmatrix} 1 & 0 \\ 0 & -1 \end{pmatrix}, \; 
c_\zeta(a_1) = \epsilon \otimes \begin{pmatrix} 0 & -i \\ i & 0 \end{pmatrix}, \;
c_\zeta(a_2) =  \epsilon \otimes\begin{pmatrix} 0 & 1 \\ 1 & 0 \end{pmatrix} 
\]
and hence $c_\zeta(\gamma) = \zeta \epsilon \otimes \begin{pmatrix} -i&0 \\ 0&-i \end{pmatrix}.$ If we multiply by $i \epsilon,$ we obtain $i\epsilon c_\zeta(\gamma) = \zeta^2 \otimes \id,$ and this endomorphism has trace $4 \zeta^2 = (2\zeta)^{[3]}.$

The fact that the trace of the tensor product of two linear maps is given by the product of the respective traces yields the statement for arbitrary $n.$ 
\end{proof}

\begin{remark} \label{rem:PolyMatrix}
Note that the entries of $c_\zeta(u)$ are polynomial in $\zeta.$ Hence for fixed $u \in \U(\hc),$ the map $\zeta \mapsto c_\zeta(u)$ is an entire holomorphic $\mathcal H$-valued function on $\CC.$
\end{remark}

Now we use the above construction to define representations of the Heisenberg--Clifford Lie supergroup.

\begin{definition} \label{def:SpinRep}
Let $\mathcal S, \mathcal H$ and $c_\zeta$ be as in Proposition~\ref{prop:AlgIso}, and $(\cdot, \cdot)_\zeta$ as in Lemma~\ref{lem:Tracec}. For every $\zeta \in \CC$ we define a representation $\pi_\zeta= (\pi_{\zeta,0}, d\pi_\zeta)$ of $\HC$ on $S$ by $\pi_{\zeta,0}(x) = e^{i x \zeta}$ for $x \in \RR$ and $d\pi_\zeta(v) = e^{i\pi/4}c_\zeta(v)$ for $v \in \hc.$

We also denote by $\pi_\zeta$ the element of $\C^\infty(\HC,\mathcal H)$ corresponding to $(\pi_{\zeta, 0},d\pi_{\zeta})$ by Proposition~\ref{prop:Rep}, which is given by $\pi_{\zeta}(u;x) = e^{i\zeta x} d\pi_{\zeta}(u)$ for $x \in \RR$ and $u \in \U(\hc).$ If $\zeta >0$ is real, then the representation $\pi_\zeta$ is unitary on the super Hilbert space $(\mathcal S, (\cdot, \cdot)_\zeta).$ 
\end{definition}

\begin{remark} \label{rem:Tracepi}
The element $\gamma$ acts by $d\pi_{\zeta}(\gamma) = (e^{i\pi/4})^n c_\zeta(\gamma)$ Together with Lemma~\ref{lem:Tracec} this implies
\[
T(d\pi_\zeta(\gamma))= (2\zeta)^{[n]},
\]
and $T(d\pi_\zeta(a_I))=0$ if $I \subsetneq \underline n.$ 
\end{remark}


\section{The Fourier Transform} \label{sct:FT}



Throughout the remainder of this article, we let $\mathcal S$ and $\mathcal H$ be as in Proposition~\ref{prop:AlgIso}, and for $\zeta \in \CC,$ we let $(\pi_\zeta, \mathcal S)$ be the representation of $\HC$ defined in~\ref{def:SpinRep}. By Proposition~\ref{prop:Rep} there is a unique element of $\C^\infty(\HC,\mathcal H)$ corresponding to $\pi_\zeta,$ which we denote by the same letter. We use the $\mathcal H$-valued function $\pi_\zeta$ to define the Fourier transform of a compactly supported smooth function on $\HC.$
\begin{definition} \label{def:FT}
If $F \in \C^\infty_c(\HC),$ we define the \emph{Fourier transform} $\mathcal F(F)$ or $\widehat F$ by 
\[
\mathcal F(F): \CC \rightarrow \mathcal H, \quad
\mathcal F(F)(\zeta) := \widehat F(\zeta) := \int_{\HC} F \cdot \pi_{-\zeta}.
\]
This is well-defined since it is the integral of a compactly supported smooth function with values in the finite-dimensional complex vector space $\mathcal H.$
\end{definition}

\begin{remark}
The Fourier transform of $F \in \C^\infty_c(\HC)$ can be computed explicitly as
\begin{align*}
\langle F,\pi_{-\zeta} \rangle &= \sum_i \int_\RR F(\gamma_i^{(1)};x)
\pi_{-\zeta}(\gamma_i^{(2)};x) \, \mathrm{d} x \\
&= \sum_i \int_\RR F(\gamma_i^{(1)};x)
 e^{-i\zeta x} d\pi_{-\zeta}(\gamma_i^{(2)}) \, \mathrm{d} x.
\end{align*}
Here, we have used that representations have even parity, and that $\pi_{-\zeta}(u;x) = \pi_{-\zeta,0}(x) d\pi_{-\zeta}(u) = e^{-i\zeta x} d\pi_{-\zeta}(u).$ We conclude that
\begin{equation} \label{eq:FT}
\widehat F(\zeta) = \sum_{i} F(\gamma_i^{(1)}) \widehat \; (\zeta) \, d\pi_{-\zeta}(\gamma_i^{(2)}),
\end{equation}
where $F(\gamma_i^{(1)}) \widehat{}\;$ is the classical Fourier transform of $F(\gamma_i^{(1)}) \in C_c^\infty(\RR).$ In particular, $\widehat F$ extends to an entire holomorphic $\mathcal H$-valued function.
\end{remark}

An immediate consequence of this remark is the following Fourier inversion formula.

\begin{proposition} \label{prop:FourierInversion}
For $F \in \C_c^\infty(\HC)$ we have
\[
F(1;x) = \frac{1}{2\pi} \int_\RR T(\widehat F(\zeta))(-2\zeta)^{-[n]} e^{i x \zeta} \, \mathrm{d} \zeta.
\]
\end{proposition}
\begin{proof}
We apply $T$ to the sum in equation~\eqref{eq:FT}.  By Remark \ref{rem:Tracepi}, only the summand with $\gamma_i^{(2)} = \gamma$ contributes. Then, $\gamma_i^{(1)} =1$ and we obtain
\[
T(\widehat F(\zeta)) = F(1) \widehat \; (\zeta) T(d\pi_{-\zeta}(\gamma)) = F(1) \widehat \; (\zeta) (-2 \zeta)^{[n]},
\]
so that the claim follows from the classical Fourier inversion formula.
\end{proof}

\begin{theorem} \label{thm:Intertwine}
The Fourier transform satisfies
\[
(L_{u;x}F) \widehat \; (\zeta) = \pi_{-\zeta}(u;x) \widehat F(\zeta)
\]
for all $x \in \RR$ and $u \in \U(\hcn).$ 
\end{theorem}
\begin{proof}
Invariance of the integral implies
\[
(L_{u;x}F) \widehat \; (\zeta) =\langle L_{u;x}F, \pi_{-\zeta} \rangle =
(-1)^{|u||F|} \langle F,L_{S(u);-x} \pi_{-\zeta} \rangle.
\]
We have 
\[
L_{S(u);-x}\pi_{-\zeta}(v,y) = \pi_{-\zeta}(uv; x+y) =  \pi_{-\zeta,0}(x+y) d\pi_{-\zeta}(uv) = \pi_{-\zeta}(u;x) \pi_{-\zeta}(v;y),
\]
by equation~\eqref{eq:Rep} and because $\pi_{-\zeta,0}(y)$ commutes with $d\pi_{-\zeta}(u).$ This implies
\[
\langle L_{u;x} F, \pi_{-\zeta} \rangle = (-1)^{|u||F|} \langle F, \pi_{-\zeta}(u;x) \pi_{-\zeta} \rangle = \pi_{-\zeta}(u;x) \langle F, \pi_{-\zeta} \rangle, 
\]
and hence the claim.
\end{proof}

\begin{remark}
Theorem~\ref{thm:Intertwine}, together with Proposition~\ref{prop:FourierInversion}, implies
\begin{align*}
F(u;x) &= (-1)^{|u||F|}(L_{S(u)}F)(1;x) \\
&= (-1)^{|u||F|}\frac{1}{2\pi} \int_\RR T(\mathcal F(L_{S(u)}F)(\zeta)) (-2\zeta)^{-[n]} e^{i x \zeta} \, \mathrm{d} \zeta \\
&= (-1)^{|u||F|}\frac{1}{2\pi} \int_\RR T(d\pi_{-\zeta}(S(u)) \widehat F(\zeta)) (-2\zeta)^{-[n]} e^{i x \zeta} \, \mathrm{d} \zeta.
\end{align*}
Recall that $d\pi(S(u))=d\pi(u)^\dagger$ for unitary representations $\pi=(\pi_0, d\pi).$ Therefore, using $\langle A|B \rangle = T(AB^\dagger)$ we can write
\begin{align*} 
F(u;x) &= (-1)^{|u||F|}\frac{1}{2\pi} \int_\RR T(d\pi_{-\zeta}(u)^{\dagger}\widehat F(\zeta)) (-2\zeta)^{-[n]} e^{i x \zeta} \, \mathrm{d} \zeta \\
&= \frac{1}{2\pi} \int_\RR T(\widehat F(\zeta) d\pi_{-\zeta}(u)^\dagger) (-2\zeta)^{-[n]} e^{i x \zeta} \, \mathrm{d} \zeta \\
&= \frac{1}{2\pi} \int_\RR \langle \widehat F(\zeta) |  d\pi_{-\zeta}(u) \rangle (-2\zeta)^{-[n]} e^{i x \zeta} \, \mathrm{d} \zeta.
\end{align*}
\end{remark}

\begin{definition}
\begin{enumerate}[a)]
\item If $A: \RR \rightarrow \mathcal H$ is any map, we define 
\[
A(u;\zeta) := \langle A(\zeta) | d\pi_{-\zeta}(u) \rangle  (-2\zeta)^{-[n]}
\]
for $u \in \U(\hc)$ and non-zero $\zeta \in \RR.$ Note that 
\[
A(zu;\zeta) = \langle A(\zeta) |  d\pi_{\zeta}(zu) \rangle  (2\zeta)^{-[n]} = (i\zeta) A(u; \zeta),
\]
since $\langle \cdot | \cdot \rangle $ is conjugate linear in the second argument.
\item We say that $A: \RR \rightarrow \mathcal H$ is of \emph{Schwartz class} if for all $u \in \U(\hc)$ and all $k \geq 0,$ the function $A(u;\zeta)$ is smooth on $\RR,$ and 
\[
s_{k,u}(A):=\sup_{\zeta \in \RR} \left| \frac{d^k}{d \zeta^k}  A(u;\zeta) \right| < \infty.
\]
The space of $\mathcal H$-valued functions of Schwartz class is denoted $\mathcal S(\RR, \mathcal H).$ 
\item If $A \in \mathcal S(\RR, \mathcal H),$ we define 
\begin{equation} \label{eq:Inversion}
\mathcal F^{-1}(A)(u;x) := \frac{1}{2\pi} \int_\RR A(u;\zeta) e^{i x \zeta}\,\mathrm{d} \zeta.
\end{equation}
\end{enumerate}
If $(a_i)_{i=1}^n$ is a basis of $V$ and we let $s_{k,j,J}:= s_{k,z^ja_J},$ then the countable family $(s_{k,j,J})$ defines a locally convex vector space topology on $\mathcal S(\RR, \mathcal H).$
\end{definition}

\begin{remark}
 If $A \in \mathcal S(\RR, \mathcal H),$ then $\mathcal F^{-1}(A)$ is an element of $\C^\infty(\HC),$ because
\[
\mathcal F^{-1}(A(zu))(x) = \mathcal F^{-1}(i\zeta A(u))(x) = \frac{\mathrm{d}}{\mathrm{d}x} \mathcal F^{-1}(A(u))(x).
\]
\end{remark}

\begin{lemma} \label{lem:SchwartzComponents}
A function $A: \RR \rightarrow \mathcal H$ is of Schwartz class if and only if for all $u \in \U(\hc),$ the function $x \mapsto A(u;x) $ is in $\mathcal S(\RR).$
\end{lemma}
\begin{proof}
Clearly, if all $A(u)$ are in $\mathcal S(\RR),$ then $s_{k,u}(A) < \infty$ for all $u \in \U(\hc), k \in \NN.$ Conversely, suppose that $A \in \mathcal S(\RR, \mathcal H).$ Fix $u \in \U(\hc).$ Then
\[
s_{k,z^ju}(A) = \sup_{\zeta \in \RR} \left| \frac{d^k}{d \zeta^k}  A(z^ju;\zeta) \right| = \sup_{\zeta \in \RR} \left| \frac{d^k}{d \zeta^k} \zeta^j  A(u;\zeta) \right| 
< \infty
\]
for all $k,j \in \NN,$ where we have used that $A(z^j u; \zeta) = (i\zeta)^j A(u;\zeta).$ It follows that $A(u) \in \mathcal S(\RR).$ 
\end{proof}

\begin{theorem} \label{thm:FourierInversion}
The Fourier transform is an isomorphism of the topological vector spaces $\mathcal S(\HC)$ and $\mathcal S(\RR, \mathcal H).$ Its inverse is given by~\eqref{eq:Inversion}.
\end{theorem}
\begin{proof}
Let $F \in \mathcal S(\HC).$ Then
\begin{align*}
s_{k,j,J}(\widehat F) &= \sup_{\zeta \in \RR} \left| \frac{d^k}{d \zeta^k}  \widehat F(z^ja_J;\zeta) \right| = \sup_{\zeta \in \RR} \left| \frac{d^k}{d \zeta^k} \langle \widehat F(\zeta), d\pi_{-\zeta}(z^ja_J) \rangle \right| \\
&= \sup_{\zeta \in \RR} \left| \frac{d^k}{d \zeta^k} \left( \zeta^j \widehat F(a_{J} ; \zeta) \right) \right|,
\end{align*}
which is finite because $\widehat F(u) \in \mathcal S(\RR)$ for all $u \in \U(\hc).$ In particular, the last expression is continuous in $F,$ which proves that $\mathcal F: \mathcal S(\HC) \rightarrow \mathcal S(\RR, \mathcal H)$ is continuous. 

If $A \in \mathcal S(\RR, \mathcal H),$ then by Lemma~\ref{lem:SchwartzComponents}, the components $A(u)$ are in $\mathcal S(\RR).$ It follows that $\mathcal F^{-1}(A) \in \mathcal S(\HC).$ We apply the seminorms $s_{j,u},$ defined in~\ref{def:SchwartzSpace}, to obtain
\[
s_{j,u}(\mathcal F^{-1}(A)) = \sup_{x \in \RR} \left| x^j \mathcal F^{-1}(A)(u;x)\right|= \sup_{x \in \RR} \left| \ \mathcal F^{-1}\left( 
\frac{\mathrm{d^j}}{\mathrm{d} \zeta^j} A(u) \right) (x) \right|.
\]
The last expression involves the classical Fourier transform of the Schwartz function $(\mathrm{d^j}/ \mathrm{d} \zeta^j)(A(u))$ and is continuous in $A(u).$ This proves the continuity of $\mathcal F^{-1}.$ 
\end{proof}


To conclude this section, we define a Fourier--Laplace transform for compactly supported distributions on $\HC$ and prove a theorem of Paley--Wiener--Schwartz type.

\begin{definition}
If $U \in \mathcal E'(\HC)$ is a compactly supported distribution, we let
\[
\widehat U(\zeta) := \langle U,\pi_{-\zeta} \rangle.
\]
for $\zeta \in \RR.$
\end{definition}

\begin{theorem} \label{thm:FourierLaplace}
If $U \in \mathcal E'(\HC)$ we can extend the definition of $\widehat U(\zeta)$ to complex values of $\zeta.$ The function $\widehat U$ is an entire holomorphic function with values in $\mathcal H.$
\end{theorem}
\begin{proof}
We write
\[
\langle U,F \rangle  = \sum_i (-1)^{|\gamma_i^{(1)}||F|} \langle U(\gamma_{i}^{(1)}), F(\gamma_{i}^{(2)}) \rangle
\]
for appropriate distributions $U(\gamma_{i}^{(1)}) \in \mathcal E'(\RR).$ Then
\[
\widehat U(\zeta) = U(\pi_{-\zeta}) = \sum_i U(\gamma_{i}^{(1)})_x (\pi_{-\zeta}(\gamma_i^{(2)};x)) = \sum_i U(\gamma_{i}^{(1)})_x(e^{-i\zeta x}) d\pi_{-\zeta}(\gamma_i^{(2)}).
\]
Here $U(\gamma_{i}^{(1)})_x(e^{-i\zeta x})$ is the classical Fourier--Laplace transform of the compactly supported distribution $U(\gamma_i^{(1)}).$ In particular, $U(\gamma_{i}^{(1)})_x(e^{-i\zeta x})$ extends to an entire holomorphic function of $\zeta$ (see~\cite[Theorem 7.1.14]{Hormander}). The same is true for the matrices $d\pi_{-\zeta}(\gamma_i^{(2)})$ by Remark~\ref{rem:PolyMatrix}, hence the claim follows.
\end{proof}

Now we formulate a  Paley--Wiener--Schwartz theorem, which characterizes the Fourier transformation of a compactly supported function by a growth condition on the `components'  $A(u)$ of $A = \mathcal F(F): \CC \rightarrow \mathcal H.$ 

\begin{definition}
a) Let $\C_{[-a,a]}^\infty(\HC)$ denote the space of functions $F \in \C_c^\infty(\HC)$ for which the support of $F(u)$ is contained in the compact interval $[-a,a]$ for all $u \in \U(\hc).$ 

b) We say that an entire holomorphic function $A: \CC \rightarrow \mathcal H$ is \emph{of exponential type $a,$} where $a$ is a positive real number, if for every $N \in \NN$ there is a constant $C_N$ such that
\begin{equation} \label{eq:growth}
|\zeta^{-[n]}T(A(\zeta) d\pi_{-\zeta}(u))| \leq C_N (1+|\zeta|)^{-N} e^{a \mathrm{Im}(\zeta)} 
\end{equation}
for all $u \in \U(\hc), \, \zeta \in \CC$.
\end{definition}

\begin{theorem} \label{thm:PaleyWienerSchwartz}
The Fourier--Laplace transform is a bijection between the space $\C_{[-a,a]}^\infty(\HC)$ and the space of entire holomorphic functions $A: \CC \rightarrow \mathcal H$ of exponential type $a.$ 
\end{theorem}
\begin{proof}
Let $F \in \C^\infty(\HC).$ We use coordinates, so that 
\[\widehat F(\zeta) = \sum_{I \subset \underline n} F(a_I) \widehat \; (\zeta) d\pi_{-\zeta}(*a_I),
\] 
and therefore
\[
|\zeta^{-[n]}T(A(\zeta) d\pi_{-\zeta}(a_I)))| = C |F(a_I) \widehat \; (\zeta)|
\]
for some positive constant $C.$ By the corresponding classical Paley--Wiener--Schwartz theorem~\cite[Theorem 7.3.1]{Hormander}, this shows that if $\mathrm{supp}(F(a_I)) \subset [-a,a],$ for all $I \subset \underline n,$ then $\mathcal F(F)$ is of exponential type $a.$

Conversely, if $A: \CC \rightarrow \mathcal H$ is of exponential type $a,$ then the classical theorem implies that the components $A(a_I)$ are the Fourier transforms of smooth functions $F(a_I)$ with support in $[-a,a].$ The $F(a_I)$ now uniquely determine an element $F \in \C^\infty_{[-a,a]}(\HC)$ with Fourier transform equal to $A,$ and this concludes the proof.
\end{proof}


\section{The Convolution Product} \label{sct:Conv}

In this section we study the convolution product for functions and distributions on $\HC.$  We prove basic properties of the convolution product, and in Theorem~\ref{thm:ConvFT} we show that the Fourier transform interchanges the convolution product and the pointwise product of $\mathcal H$-valued functions. Lastly, we obtain a Banach convolution algebra as the completion of $\C^\infty_c(\HC)$ with respect to a Sobolev norm.

\begin{definition} \label{def:Convolution}
Let $F, G \in \C^\infty(\HC)$ and assume that one of $F,G$ is compactly supported. We define the \emph{convolution of $F$ and $G$} by 
\[
(F*G)(u;x) := (-1)^{|u|(|G|+ |\gamma|)}\langle F, L_{u;x}i^*G \rangle.
\]
\end{definition}

\begin{proposition} \label{prop:ConvBasics}
\begin{enumerate}[a)]
\item Suppose that one of $F, G \in \C^\infty(\HC)$ has compact support. Then the convolution product $F*G$ is in $\C^\infty(\HC).$ 
\item $\supp(F*G) \subset \supp \, F+\supp \, G$
\item $L_{u;x}(F*G) = (L_{u;x}F)*G$ for all $u \in \U(\hc), \; x \in \RR.$
\item  \begin{equation*}
(F*G)(u) = \sum_i (-1)^{|\gamma_i^{(1)}||\gamma_i^{(2)}|} F(\gamma_i^{(1)})*G(S(\gamma_i^{(2)})u).
\end{equation*}
\end{enumerate}
\end{proposition}
\begin{proof}
To check that $F*G \in \C^\infty(\HC)$ we need to verify $\z$-linearity. This follows easily from
\[
F(\gamma_i^{(1)}) * G(S(\gamma_i^{(2)})zu) = F(\gamma_i^{(1)}) * G'(S(\gamma_i^{(2)})u)
= (F(\gamma_i^{(1)}) * G(S(\gamma_i^{(2)})u))'.
\]

The inclusion of supports in $b)$ follows easily as in the classical case.

Statement $c)$ follows easily from the invariance of the pairing:
\begin{align*}
L_{u;x}(F*G)(v;y) &= (-1)^{|u|(|F|+|G|+|\gamma|)} (F*G)(S(u)v; y-x) \\
&= (-1)^{|u||F|+|v|(|G|+|\gamma|)} \langle F, L_{S(u)v; y-x} i^*G \rangle \\
&= (-1)^{|v|(|G|+|\gamma|)} \langle L_{u;x}F, L_{v;y}i^* G \rangle \\
&= ((L_{u;x}F)*G)(v;y).
\end{align*}

We compute $F*G$ as follows.
\begin{align*}
(F*G)&(u;x) =\\
&= (-1)^{|u|(|G|+|\gamma|)} \sum_i (-1)^{|\gamma_i^{(1)}|(|u|+|G|)} \int_\RR F(\gamma_i^{(1)};y) L_{u;x}i^*G (\gamma_i^{(2)};y) \, \mathrm{d}y \\
&= (-1)^{|u||\gamma|} \sum_i (-1)^{|\gamma_i^{(1)}|(|u|+|G|)} \int_\RR F(\gamma_i^{(1)};y)(i^*G)(S(u)\gamma_i^{(2)};y-x) \, \mathrm{d}y \\
&= (-1)^{|u||\gamma|} \sum_i (-1)^{|u||\gamma|+|\gamma_i^{(1)}||G|} \int_\RR F(\gamma_i^{(1)};y)G(S(\gamma_i^{(2)})u;x-y) \, \mathrm{d}y\\ 
&= \sum_i (-1)^{|\gamma_i^{(1)}||G|}(F(\gamma_i^{(1)})*G(S(\gamma_i^{(2)})u))(x)
\end{align*}
Now the claim follows, since the summands are non-zero only if $|G| = |\gamma_i^{(2)}|.$
\end{proof}

\begin{remark} \label{rem:Coord}
Choose coordinates and consider the elements $F = f \otimes \xi^I$ and $G = g \otimes \xi^J$ of $\C^\infty_c(\HC).$ We compute $F * G$ using Proposition~\ref{prop:ConvBasics} b).
\begin{align*}
(F*G)(a_K) &= \sum_{L \subset \underline n} (-1)^{|L||L^c|} F(a_L)*G(*a_L a_K) \\ 
&= (-1)^{|I||I^c|} f* G(*a_I a_K).
\end{align*}
This is non-zero only if $*a_I a_k$ is a multiple of $a_J.$ Now $*a_I = \mathrm{sgn}(\sigma_I) a_{I^c},$ and in $\U(\hc)$ we have the equality
\[
a_{I^c} a_{(I \Delta J)^c} = \pm z^{|I^c \cap J^c|} a_J \in \U(\hc),
\]
where $I \Delta J$ is the symmetric difference of the subsets $I, J \subset \underline n.$ Hence, we choose $K = (I \Delta J)^c$ and obtain
\begin{equation*}
(f \otimes \xi^I) * (g \otimes \xi^J) = \pm (f*g)^{(|I^c \cap J^c|)} \otimes \xi^{(I \Delta J)^c}.
\end{equation*}
\end{remark}

\begin{proposition} \label{prop:ConvBasic}
The convolution product is an $\RR$-bilinear, continuous  map $*: \C^\infty(\HC) \times \C^\infty_c(\HC) \rightarrow \C^\infty(\HC)$ with parity $|\gamma|,$ that is, $|F*G| = |F|+|G|+|\gamma|.$
\end{proposition}
\begin{proof}
Bilinearity is clear, and continuity can be conveniently checked in coordinates. By Remark~\ref{rem:Coord} we have
\[
(f \otimes \xi^I, g \otimes \xi^J) \mapsto \pm (f*g)^{(|I^c \cap J^c|)} \otimes \xi^{(I \Delta J)^c}, 
\]
hence continuity follows from the corresponding result for functions on $\RR,$ together with continuity of the derivative as a map from $C^\infty(\RR)$ to $C^\infty(\RR).$ The parity of the convolution product equals the parity of the pairing $ \langle \cdot, \cdot \rangle ,$ which is $|\gamma|$ by Remark~\ref{rem:ParityPairing}.
\end{proof}

\begin{proposition} \label{prop:ConvExist}
Let $F,G$ and $\Phi$ be smooth functions on $\HC,$ at least two of which are compactly supported. Then
\begin{equation*}
\langle F*G,\Phi \rangle =  \langle F \otimes G, m^* \Phi \rangle
\end{equation*}
and
\begin{equation*}
\langle F*G,\Phi \rangle = \langle F, i^*(G* i^*\Phi) \rangle.
\end{equation*}
Furthermore, if one of $F,G$ is compactly supported, then
\begin{equation*}
F*G = (-1)^{|F||G|}i^*(i^*G * i^*F).
\end{equation*}
\end{proposition}
\begin{proof}
We compute $ \langle F*G,\Phi \rangle $ for $\Phi \in \C^\infty_c(\HC)$ as
\begin{align*}
&\langle F*G,\Phi \rangle = \sum_j (-1)^{|\Phi||\gamma_j^{(1)}|} \int_\RR(F*G)(\gamma_j^{(1)} ;y) \Phi(\gamma_j^{(2)};y) \, \mathrm{d} y \\
&= \sum_{i,j} (-1)^{|\Phi||\gamma_j^{(1)}| + |G||\gamma_i^{(1)}|}
\int_\RR (F(\gamma_i^{(1)})*G(S(\gamma_i^{(2)}) \gamma_j^{(1)}))(y) \Phi(\gamma_j^{(2)};y) \, \mathrm{d} y \\
&= \sum_{i,j} (-1)^{|\Phi||\gamma_j^{(1)}| + |G||\gamma_i^{(1)}|}
\int_{\RR^2} F(\gamma_i^{(1)};x) G(S(\gamma_i^{(2)})\gamma_j^{(1)};y-x) \Phi(\gamma_j^{(2)};y) \, \mathrm{d}x \mathrm{d}y  \\
&= \sum_{i,j} (-1)^{|\Phi||\gamma_j^{(1)}| + |G||\gamma|}
\int_{\RR^2} F(\gamma_i^{(1)};x) (L_{\gamma_i^{(2)};x}G)(\gamma_j^{(1)};y) \Phi(\gamma_j^{(2)};y) \, \mathrm{d}x \mathrm{d}y  \\
&= \sum_i (-1)^{|G||\gamma|}
\int_{\RR} F(\gamma_i^{(1)};x) (L_{\gamma_i^{(2)};x}G, \Phi) \, \mathrm{d}x.
\end{align*}
Next, we use the invariance of the pairing $\langle \cdot, \cdot \rangle$ to obtain
\begin{align*}
\langle F*G, \Phi \rangle  &= \sum_i (-1)^{|G||\gamma_i^{(1)}|}
\int_{\RR} F(\gamma_i^{(1)};x) (G, L_{S(\gamma_i^{(2)});-x}\Phi) \, \mathrm{d}x \\
&= \sum_{i,j} (-1)^{|G||\gamma_i^{(1)}| + |\Phi|(|\gamma_i^{(1)}|+|\gamma_j^{(1)}|) + |\gamma_j^{(1)}||\gamma_i^{(2)}|} \\ & \hskip1cm \int_{\RR^2} F(\gamma_i^{(1)};x) G(\gamma_j^{(1)};x) \Phi(\gamma_i^{(2)}\gamma_j^{(2)}; x+y) \, \mathrm{d}x \mathrm{d} y \\
&= \langle F \otimes G, m^*\Phi \rangle.
\end{align*}
The function $i^*(G * i^* \Phi)$ takes values 
\begin{align*}
i^*(G * i^* \Phi)(v;y) &= (G*i^*\Phi)(S(v);-y) \\
&= \sum_j (-1)^{|\Phi||\gamma_j^{(1)}|} (G(\gamma_j^{(1)}) * (i^*\Phi)(S(\gamma_j^{(2)})S(v)))(-y) \\
&=  \sum_j (-1)^{|\Phi||\gamma_j^{(1)}|} \int_\RR
G(\gamma_j^{(1)};x)(i^*\Phi)(S(\gamma_j^{(2)})S(v))(-x-y) \, \mathrm{d} x \\
&=  \sum_j (-1)^{|\Phi||\gamma_j^{(1)}|} \int_\RR
G(\gamma_j^{(1)};x)\Phi(\gamma_j^{(2)}v; x+y) \, \mathrm{d} x.
\end{align*}
Therefore, comparing with the above computations yields
\begin{align*}
\langle F*G, \Phi \rangle &= \sum_i (-1)^{|\gamma_i^{(1)}|(|G|+|\Phi|)} \int_\RR F(\gamma_i^{(1)};x) i^*(G* i^*\Phi)(\gamma_i^{(2)};x) \, \mathrm{d} x \\
&= \langle F, i^*(G* i^* \Phi) \rangle. 
\end{align*}

Lastly, we compute
\begin{align*}
(F*G)(u;x) &= (-1)^{|u|(|G|+|\gamma|)} \langle F,L_{u;x}i^*G \rangle \\
&= (-1)^{|u|(|G|+|\gamma|) + |F|(|u|+|G|)} \langle L_{u;x}i^*G, F \rangle \\
&= (-1)^{|F||G| + |u|(|F|+ |\gamma|)} \langle i^*G, L_{S(u);-x}F \rangle \\
&= (-1)^{|F||G|}(i^*G * i^*F)(S(u);-x),
\end{align*}
which proves that $F*G = (-1)^{|F||G|}i^*(i^*G * i^* F).$
\end{proof}

For later use we record the following corollary of our proof.

\begin{corollary} \label{cor:Conv}
\begin{enumerate}[a)]
\item
\[
\langle F*G, \Phi \rangle = \langle F, i^*(G* i^* \Phi) \rangle =  \sum_i (-1)^{|G||\gamma|}
\int_{\RR} F(\gamma_i^{(1)};x) (L_{\gamma_i^{(2)};x}G, \Phi) \, \mathrm{d}x.
\]
\item The convolution product is associative.
\end{enumerate}
\end{corollary}
\begin{proof}
Equality a) occurs in the course of the proof, and b) follows because
\[
(F,G) \mapsto \langle F \otimes G, m^*\Phi \rangle
\]
is associative.
\end{proof}

\begin{theorem} \label{thm:ConvFT}
If $F,G \in \mathcal \C^\infty_c(\HC)$ and $\zeta \in \RR,$ then
\[
(F*G) \widehat \; (\zeta) = \widehat F(\zeta) \widehat G(\zeta)
\]
in $\mathcal H.$
\end{theorem}
\begin{proof}
By definition of the Fourier transform and Proposition~\ref{prop:ConvExist},
\[
(F*G) \widehat \; (\zeta) = \langle F*G , \pi_{-\zeta} \rangle = \langle F \otimes G , (m^* \otimes \id_S) \circ \pi_{-\zeta} \rangle.
\]
By Proposition~\ref{prop:Rep} we have  $(m^* \otimes \id_S) \circ \pi_{-\zeta} = (\id_S \otimes \pi_{-\zeta}) \circ  \pi_{-\zeta},$ and the latter is given by
\[
((\id_S \otimes \pi_{-\zeta}) \circ  \pi_{-\zeta})( u \otimes v; x,y) = \pi_{-\zeta}(u;x) \circ  \pi_{-\zeta}(v;y) \in \mathcal H.
\]
From this we conclude $\langle F \otimes G , (m^* \otimes \id_S) \circ \pi_{-\zeta} \rangle = \langle F, \pi_{-\zeta} \rangle \langle G, \pi_{-\zeta} \rangle $ and hence the claim.
\end{proof}

For completeness, we also compute the Fourier transform of the pointwise product of compactly supported superfunctions. To this end, we use the notation
\[
((\Delta \otimes \id) \circ \Delta)(\gamma) = \sum_i \gamma_i^{(1)} \otimes \gamma_i^{(2)} \otimes \gamma_i^{(3)}.
\]

\begin{proposition}
If $F,G \in \C^\infty_c(\HC),$ then the Fourier transform of $F\cdot G$ is given by 
\[
(F \cdot G) \widehat \; (\zeta)= \frac{1}{2\pi} \sum_i (F(\gamma_i^{(1)}) \widehat \;   * G(\gamma_i^{(2)}) \widehat \;  )(\zeta) d\pi_{-\zeta}(\gamma_i^{(3)}).
\]
\end{proposition}
\begin{proof}
We use the definition of Fourier transform and of the product in $\C^\infty(\HC)$ to obtain
\begin{align*}
(F \cdot G) \widehat \; (\zeta) &= \langle F \cdot G, \pi_{-\zeta} \rangle = \langle \mu \circ F \otimes G \circ \Delta, \pi_{-\zeta} \rangle \\
&= \int_\RR (F \otimes G \otimes \pi_{-\zeta}) (\Delta \circ (\id \otimes \Delta)(\gamma;x) \, \mathrm{d} x \\
&= \sum_i \mathcal F(F(\gamma_i^{(1)}) * G(\gamma_i^{(2)}))(\zeta) d\pi_{-\zeta}(\gamma_i^{(3)}).
\end{align*}
Now the claim follows from the classical fact
\[
\widehat{f g} (\zeta)  = \frac{1}{2\pi} (\hat f * \hat g)(\zeta).
\]
\end{proof}

\subsection*{Convolution with a Distribution}

Our next goal is to define the convolution $U*F$ of a distribution $U \in \mathcal D'(\HC)$ and a compactly supported function $F.$

\begin{definition}
Let $U \in \mathcal D'(\HC)$ and $F \in \C_c^\infty(\HC).$ We define a linear functional $U*F$ on $\C^\infty(\HC)$ by 
\[
\langle U*F, \Phi \rangle := \langle U, i^*(F* i^* \Phi) \rangle.
\]
\end{definition}

\begin{remark}
Note that by Proposition~\ref{prop:ConvExist}, for $U \in \C^\infty(\HC)$ this definition agrees with Definition~\ref{def:Convolution}.
\end{remark}

\begin{proposition}
Given $U \in \mathcal D'(\HC)$ and $F \in \C_c^\infty(\HC),$ the functional $U*F$ is a distribution on $\HC.$ The distribution $U*F$ is given by the smooth function $H \in \C^\infty(\HC)$ defined by 
\begin{equation} \label{eq:ConvDist}
H(u) = \sum_i (-1)^{|\gamma_i^{(1)}||F|}U(\gamma_i^{(1)})*F(\gamma_i^{(2)}u) 
\end{equation}
for $u \in \U(\hc).$
\end{proposition}
\begin{proof}
We need to show that
\[
\langle U , i^*(F*i^*\Phi) \rangle =  \langle H,\Phi \rangle
\]
holds for all $\Phi  \in \C^\infty_c(\HC).$ First we compute $ \langle H, \Phi \rangle $ as
\begin{align*}
\langle H,\Phi \rangle &= \sum_j (-1)^{|\gamma_j^{(1)}||\Phi|}\int_\RR H(\gamma_j^{(1)};y) \Phi(\gamma_j^{(2)};y) \, \mathrm{d}y \\
&= \sum_{i,j} (-1)^{|\gamma_j^{(1)}||\Phi| + |\gamma_i^{(1)}||F|}
\int_\RR (U(\gamma_i^{(1)})*F(S(\gamma_i^{(2)})\gamma_j^{(1)}))(y) \Phi(\gamma_j^{(2)};y) \, \mathrm{d} y \\
&= \sum_i (-1)^{|\gamma_i^{(1)}||F|} \\
&\hskip1cm U(\gamma_i^{(1)}) \left( \sum_j(-1)^{|\gamma_j^{(1)}||\Phi|} \int_\RR F(S(\gamma_i^{(2)})\gamma_j^{(1)})(y-x) \Phi(\gamma_j^{(2)};y) \, \mathrm{d} y \right) \\
&= \sum_i (-1)^{|F||\gamma|} \langle U(\gamma_i^{(1)}), \langle L_{\gamma_i^{(2)};x}F, \Phi \rangle \rangle.
\end{align*}
Here we have used that 
\[
\int_{\RR} (u*f)(y) \phi(y) \, \mathrm{d} y = u_x\left( \int_\RR f(y-x) \phi(y) \, \mathrm{d} y \right)
\]
for $u \in \mathcal D'(\RR)$ and $f,\phi \in C^\infty_c(\RR).$ Comparing with Corollary~\ref{cor:Conv} we obtain
\[
\langle H,\Phi \rangle = \langle U , i^*(F*i^*\Phi) \rangle = \langle U*F, \Phi\rangle,
\]
which completes the proof.
\end{proof}

\begin{example}
Let $e=(e_0,e^*)$ be the identity of $\HC$ and define a distribution $U$ via $\langle U, \Phi \rangle:= e^*\Phi = \Phi(1;0).$ Then $U \in \mathcal E'(\HC),$ which can be shown easily in coordinates. We compute
\begin{align*}
\langle U*F,\Phi \rangle &= \langle U , i^*(F*i^*\Phi) \rangle = (F*i^*\Phi)(1;0) \\
&= \sum_i (-1)^{|\gamma_i^{(1)}||\Phi|} \int_\RR F(\gamma_i^{(1)};x)(i^*\Phi)(S(\gamma_i^{(2)});-x) \, \mathrm{d} x \\
&=  \sum_i (-1)^{|\gamma_i^{(1)}||\Phi|} \int_\RR F(\gamma_i^{(1)};x)\Phi(\gamma_i^{(2)};x) \, \mathrm{d} x  = \langle F, \Phi \rangle
\end{align*}
for all $\Phi \in \C_c^\infty(\HC).$ Therefore, $U*F = F$ for all $F \in \C^\infty(\HC).$ 
\end{example}

\subsection*{The Convolution Algebra}

We begin by recalling the definition of Sobolev spaces on $\RR.$ We refer to~\cite{Adams} for details. The \emph{Sobolev space $W^{k,p}(\RR)$} is the space of functions in $L^p(\RR)$ whose distributional derivatives up to order $k$ exist and are in $L^p(\RR).$ The space $W^{k,p}(\RR)$ is a Banach space with the norm
\[
\|f\|_{k,p} := \left( \sum_{0 \leq j \leq k} \|f^{(j)} \|_p^p \right)^{1/p},
\]
where $f^{(j)}$ denotes the $j$th weak or distributional derivative of $f.$

For the rest of this section we fix an orthonormal basis $(a_i)_{i=1}^n$ of $V.$ We define Sobolev norms on $\C^\infty_c(\HC(V))$ in analogy with the definition of $\| \cdot \|_{k,p}$ above, replacing the derivatives by the differential operators $L_{z^ja_J}$ for $z^j a_J \in \U(\hc).$ As in the classical case, a different choice of basis $(a_i)_{i=1}^n$ will lead to an equivalent norm.

\begin{definition}
If $1 \leq p < \infty$ and $k \in \NN,$ we define a seminorm $\|\cdot \|_{k,p}$ on $\C^\infty_c(\HC)$ by 
\[
\| F \|_{k,p} = \left( \sum_{j+(\#I) \leq k} \| (L_{z^j a_I}F)(1) \|_p^p\right)^{1/p}.
\]
Here $\#I$ denotes the cardinality of $I \subset \underline n.$
\end{definition}

\begin{lemma}
If $k \geq \dim V,$ then $\|\cdot \|_{k,p}$ is a norm on $\C^\infty_c(\HC).$ The completion of $\C^\infty_c(\HC)$ with respect to $\| \cdot \|_{k,p}$ is a Banach space isomorphic to
\[
\bigoplus_{I \subset \underline n} W^{k-(\# I),p}(\RR),
\]
equipped with the norm
\[
\| (f_I)_{I \subset \underline n} \| = \left( \sum_{I \subset \underline n} \|f_I\|_{k-(\# I),p}^p \right)^{1/p}.
\]
\end{lemma}
\begin{proof}
Recall that after the choice of basis $(a_i)_{i=1}^n$ of $V$ we can identify 
\[
\C_c^\infty(\HC) \cong \bigoplus_{I \subset \underline n} C_c^\infty(\RR)
\]
as vector spaces via $F \mapsto \sum_{I \subset \underline n} f_I \otimes \xi^I,$ where $f_I = F(a_I).$ If $F = f \otimes \xi^I$ and $k \geq \# I,$ then $\| F \|^p_{k,p}$ is given by 
\[
\sum_{j+(\# J) \leq k} \|(L_{z^j a_J} F)(1) \|_p^p = \sum_{j+(\# J) \leq k} \| F(z^j a_J) \|_p^p = \sum_{j \leq k-(\# I)} \|f^{(j)} \|_p^p,
\]
hence
\[
\| f \otimes  \xi^I \|_{k,p} = \| f \|_{k- (\# I),p}.
\]
Clearly, if $k < |I|,$ then $\|f \otimes \xi^I\|_{k,p} = 0,$ which shows that $\| \cdot \|_{k,p}$ is a norm on $\C_c^\infty(\HC)$ only for $k \geq \dim V.$ The completion of $\C^\infty_c(\HC)$ is the direct sum of the completions of $(C_c^\infty(\RR), \|\cdot \|_{k-(\# I),p}).$ But these completions are precisely the classical Sobolev spaces $(W^{k-(\# I),p}(\RR), \| \cdot \|_{k-(\# I),p}),$ see~\cite[Theorem 3.23]{Adams}. 
\end{proof}

\begin{definition}
We denote the completion of $\C^\infty_c(\HC)$ with respect to  $\| \cdot \|_{k,p}$ by $W^{k,p}(\HC).$
\end{definition}

\begin{theorem} \label{thm:BanachAlgebra}
If $\HC = \HC(V)$ with $\dim V = n,$ then $(W^{n,1}(\HC),*)$ is a Banach algebra.
\end{theorem}
\begin{proof}
Let $(a_i)_{i=1}^n$ be the basis used to define the norm $\|\cdot \|_{n,1},$ and let $F,G \in W^{n,1}(\HC).$ Then
\[
\| F*G \|_{n,1} = \sum_{j+(\# J) \leq n} \| F*G(z^ja_J) \|_1 \leq \sum_{j+(\# J) \leq n} \sum_{I \subset \underline n} \| F(a_I)*G(z^ja_{I^c} a_J)\|_1,
\]
by definition of the convolution product and the triangle inequality. Writing $C := I \cap J$ and $A:= I \setminus C, \, B = J \setminus C$ we can express the last sum as a sum over pairwise disjoint subsets $A,B$ and $C$ of $\underline n.$ In the following, the prime indicates that $A,B,C$ are pairwise disjoint, and we set $a= \# A , b= \# B$ and $c = \# C.$
\begin{align*}
\| F*G \|_{n,1} &\leq \sideset{}{'}\sum_{A,B,C} \; \sum_{j \leq n-b-c} \| (F(a_{A \cup C})*G(a_{(A \cup C)^c} a_{B \cup C}))^{(j)}\|_1  \\
&=  \sideset{}{'}\sum_{A,B,C} \; \sum_{j \leq n-b-c} \| (F(a_{A \cup C})*G(a_{(A \cup B)^c}))^{(j+b)} \|_1.
\end{align*}
Now we need to `distribute' the $(j+b)$-th derivative over the two factors in the convolution product. The first factor is in $W^{n-a-c,1}(\RR)$ and the second is in $W^{a+b,1}(\RR).$ Clearly, we can differentiate the second factor $b$ times. Then the factors are differentiable of order $n-a-c$ and $a,$ respectively. But $n-a-c + a = n-c \geq j+b \geq j,$ since $j+b+c \leq n.$ Hence we can distribute the remaining $j$ derivatives over the two factors. This shows that the last sum is less than or equal to 
\[
\left( \sum_{i+(\# I) \leq n} \|F(z^i a_I) \|_1  \right) \left( \sum_{j+(\# J) \leq n} \| G(z^j a_J)\|_1 \right) = \|F\|_{n,1} \|G \|_{n,1}.
\]
\end{proof}


\vskip2cm
{\sc Alexander Alldridge \\ Mathematisches Institut, Universit\"at zu K\"oln, \\ Weyertal 86-90, 50931 K\"oln, Germany.} \\
{\it email adress:}  {\tt alldridg@math.uni-koeln.de}
\vskip1cm
{\sc Joachim Hilgert \\ Institut f\"ur Mathematik, Universit\"at  Paderborn, \\ Warburger Str. 100, 33098 Paderborn, Germany.} \\
{\it email adress:}  {\tt hilgert@math.uni-paderborn.de}
\vskip1cm
{\sc Martin Laubinger  \\ Institut f\"ur Mathematik, Universit\"at  Paderborn, \\ Warburger Str. 100, 33098 Paderborn, Germany.} \\
{\it email adress:}  {\tt mlaubing@math.uni-paderborn.de}

\end{document}